\documentclass[11pt]{amsart}
\usepackage{graphicx}
\usepackage[left=32mm,right=32mm,top=30mm,bottom=32mm]{geometry}

\usepackage{times}

\usepackage{float}
\usepackage{geometry,graphicx,fancybox}
\usepackage{enumerate,amsmath,amssymb,amsthm}
\usepackage{comment}
\usepackage{epsfig}
\usepackage{pst-grad}
\usepackage{pst-plot}
\usepackage{mathtools}
\usepackage[shortlabels]{enumitem}

\theoremstyle{plain}

\usepackage[kerning=true]{microtype} 
\usepackage{graphicx}
\usepackage{wrapfig}
\usepackage{dsfont}
\usepackage{multirow}
\usepackage{stmaryrd}

\newtheorem{theorem}{Theorem}[section]
\newtheorem{corollary}[theorem]{Corollary}
\newtheorem{lemma}[theorem]{Lemma}
\newtheorem{conjecture}[theorem]{Conjecture}
\newtheorem{proposition}[theorem]{Proposition}

\theoremstyle{remark}

\theoremstyle{definition}
\newtheorem{definition}[theorem]{Definition}
\newtheorem{remark}[theorem]{Remark}

\newcommand{\Z}{\mathbb{Z}}
\newcommand{\R}{\mathbb{R}}

\newcommand{\N}{\mathbb{N}}

\newcommand{\eps}{\varepsilon}

\DeclareMathOperator{\vol}{Vol}

\DeclareMathOperator{\area}{Area}
\DeclareMathOperator{\length}{length}

\DeclareMathOperator{\sys}{sys}

\DeclareMathOperator{\Det}{Det}

\numberwithin{equation}{section}

\title[Length products for manifolds]{Minimal length product over homology bases of manifolds}

\author{Florent Balacheff, Steve Karam \& Hugo Parlier}

\address{F. Balacheff, Universitat Aut\`onoma de Barcelona, Spain.}

\email{fbalacheff@mat.uab.cat}

\address{S. Karam, Lebanese University, Lebanon.}

\email{karam.steve.work@gmail.com}

\address{H. Parlier, University of Luxembourg, Luxembourg.}

\email{hugo.parlier@uni.lu}

\thanks{The first author acknowledges support by grants ANR-12-BS01-0009-02 and Ram\'on y Cajal RYC-2016-19334. The second author acknowledges support from grant ANR CEMPI (ANR-11-LABX-0007-01).
The third author acknowledges support from the ANR/FNR project SoS, INTER/ANR/16/11554412/SoS, ANR-17-CE40-0033.}



\begin{document}
\maketitle

\begin{abstract} 
Minkowski's second theorem can be stated as an inequality for $n$-dimensional flat Finsler tori relating the volume and the minimal product of the lengths of closed geodesics which form a homology basis. In this paper we show how this fundamental result can be promoted to a principle holding for a larger class of Finsler manifolds. This includes manifolds for which first Betti number and dimension do no necessarily coincide, a prime example being the case of surfaces. This class of manifolds is described by a non-vanishing condition for the hyperdeterminant reduced modulo $2$ of the multilinear map induced by the fundamental class of the manifold on its first $\Z_2$-cohomology group using the cup product.
 \end{abstract}

\bigskip


\section{Introduction}%
Minkowski's second theorem asserts that the product of the successive minima of a symmetric convex body in Euclidean space is bounded in terms of its volume. More precisely, if $n$ denotes a positive integer, the successive minima of a given symmetric convex body $K$ of $\R^n$, defined for $k=1,\ldots,n$ as the numbers 
$$
\lambda_k:=\min \{ t\mid tK \, \, \text{contains} \, \, k \, \, \text{linearly-independent vectors of $\Z^n$}\},
$$
satisfy the sharp inequality
$$
\prod_{k=1}^{n} \lambda_k \cdot \vol(K) \leq 2^n.
$$
If we analyze the situation in terms of metric geometry, this inequality can be reformulated as the following upper bound on a {\it length product}: any flat Finsler torus of dimension $n$ with Busemann-Hausdorff volume $V$ admits a family of closed geodesics $(\gamma_1,\ldots,\gamma_n)$ inducing a basis of its first real homology group---in short, an $\R$-homology basis---whose product of lengths satisfies
\begin{equation}\label{eq:Mink}
\prod_{k=1}^{n} \ell(\gamma_k) \leq {2^n \over b_n} \cdot V.
\end{equation}
Indeed symmetric convex bodies naturally parametrize (up to isometry) the space of flat Finsler tori through the correspondence $K \mapsto (\R^n/\Z^n, \|\cdot\|_K)$. Here $\|\cdot\|_K$ denotes the norm induced by the symmetric convex $K$. In this setting, each successive minima $\lambda_k$ is realized as the Finsler length of some closed geodesic $\gamma_k$ such that the real homology classes of these loops induce a $\R$-homology basis. Busemann-Hausdorff volume is defined by integrating the multiple of Lebesgue measure in each tangent space for which $K$---the unit ball in each tangent space of the Finsler metric---has volume equal to the Lebesgue measure (denoted by $b_n$) of the Euclidean unit ball.

\bigskip

The purpose of this article is to study analogs of this length product inequality for Riemannian and Finsler manifolds. Of course, a first natural direction is to ask whether inequality (\ref{eq:Mink}) still holds for non-flat tori. Indeed, in \cite{BK16} the first two authors proved that Riemannian $n$-dimensional tori with unit volume have their minimal length product among $\Z_2$-homology basis induced by $n$ closed geodesics uniformly bounded from above by $n^n$, answering a question by Gromov (see \cite[p.339]{Gro96}). Using comparison procedures explained in \cite[Section 4.4]{ABT16}, it implies that inequality (\ref{eq:Mink}) still holds for non-flat Finsler tori---albeit with a weaker constant---even for non-reversible ones if we use the Holmes-Thompson definition of volume instead of the Busemann-Hausdorff one. The question of the best constant in such length product inequalities for Finsler or Riemannian tori is still open. 

\bigskip

In a perhaps surprising direction, we present here a length product inequality similar to (\ref{eq:Mink}) but valid for closed manifolds whose first $\Z_2$-Betti number does not necessarily coincide with the dimension. 
Let $M$ be a closed connected manifold of dimension $n$ (and from now on all our manifolds will be connected and closed). Denote by $H^1(M;\Z_2)$ the first cohomology group of $M$ with $\Z_2$-coefficients and suppose that its dimension $b$ is positive. The $\Z_2$-fundamental class of $M$ induces a $n$-multilinear symmetric form on $H^1(M;\Z_2)$ denoted by $F_M$ as follows:
\begin{eqnarray*}
F_M: H^1(M;\Z_2) \times \ldots \times H^1(M;\Z_2) & \to & \Z_2\\
(\beta_1,\ldots,\beta_n) &\mapsto & \beta_1\cup \ldots \cup \beta_n [M].
\end{eqnarray*}
Among the algebraic invariants of multilinear forms over complex vector spaces, the {\it hyperdeterminant} denoted by $\Det$ was introduced by Cayley \cite{Cay45} and extensively studied by Gelfand, Kapranov and Zelevinsky \cite{GKZ92}. We consider its reduction modulo $2$ which we denote by $\Det_2$ (see Section \ref{sec:alginv} for more details).

\begin{theorem}\label{manifolds.Z_2.intro}
Consider a Riemannian manifold $M$ of first $\Z_2$-Betti number $b>0$ and of dimension $n$. Suppose that 
$$
\Det_2 F_M \neq 0.
$$
Then there exists a $\Z_2$-homology basis induced by closed geodesics $\gamma_1,\ldots,\gamma_{b}$ whose length product satisfies the following inequality:
\begin{equation}\label{eq:prod.intro}
\prod_{k=1}^{b} \ell(\gamma_k)\leq {n^b} \cdot \vol(M)^{b/ n}.
\end{equation}
\end{theorem}

Applying this theorem to $M$ diffeomorphic to the $n$-dimensional torus, we recover the result of \cite{BK16} that Riemannian $n$-dimensional tori with unit volume have their minimal length product among $\Z_2$-homology basis uniformly bounded from above by $n^n$.
Minkowski's second theorem thus appears as a {\it principle} which holds even when the first Betti number and the dimension do not necessarily coincide. Again, by comparison procedures explained in \cite[Section 4.4]{ABT16}, Theorem \ref{manifolds.Z_2.intro} implies that a similar inequality to (\ref{eq:prod.intro}) holds---albeit with a worst constant---for Finsler manifolds $M$ with $\Det_2 F_M \neq 0$.\\

Here are some direct applications of Theorem \ref{manifolds.Z_2.intro}. First we derive the following Minkowski's second principle for surfaces.

\begin{corollary}\label{surfaces.Z_2.intro}
Consider a Riemannian surface $S$ of first $\Z_2$-Betti number $b>0$.

Then there exists a $\Z_2$-homology basis induced by closed geodesics $\gamma_1,\ldots,\gamma_{b}$ whose length product satisfies the following inequality:
$$
\prod_{i=1}^{b} \ell(\gamma_i)\leq {2^b} \cdot \area(S)^{b/ 2}.
$$
\end{corollary}

Indeed we will check in Section \ref{sec:manifolds} that for such a surface $S$ the reduction modulo $2$ of the hyperdeterminant of $F_S$ is always non-zero.\\

When $m=3$ and $b=2$ the hyperdeterminant was explicitly described by Cayley (see \cite{Cay45}), and so using this, we are also able to deduce the following from Theorem \ref{manifolds.Z_2.intro}:

\begin{corollary}\label{connectedsum.Z_2.intro}
Consider a Riemannian manifold $M$ diffeomorphic to $\R P^3 \# \R P^3$.

Then there exists a $\Z_2$-homology basis induced by two closed geodesics $\gamma_1,\gamma_2$ whose length product satisfies the following inequality:
$$
\ell(\gamma_1) \cdot \ell(\gamma_2)\leq {9} \cdot \vol(M)^{2/3}.
$$
\end{corollary}

To our knowledge and with the exception of the $3$-dimensional torus, this is the only example of an inequality in dimension higher than 2 which bounds from below the volume of a Riemannian manifold by the lengths of several closed geodesics.\\

Using the notion of homological asymptotic volume introduced in \cite{Bab92} and its connection to the stable norm (see Section \ref{sec:stab} for more details), we derive the following asymptotic estimate from Theorem \ref{manifolds.Z_2.intro}.

\begin{corollary}\label{nb.geodesics.intro}
Consider a Riemannian manifold $M$ of first $\Z_2$-Betti number $b>0$ and of dimension $n$. Suppose that its integral homology has no $2$-torsion and that 
$$
\Det_2 F_M \neq 0.
$$
Then the number $N_M(t)$ of homologically distinct closed geodesics of length $\leq t$ satisfies the following inequality:
$$
N_M(t)\geq {2^b \over \vol(M)^{b/n} \cdot n^b \cdot b!}\, t^b +o(t^{b-1}).
$$
\end{corollary}

This improves a previous lower bound due to Babenko \cite[Theorem 7.1.a]{Bab92}. 
Note that the condition that integral homology has no $2$-torsion is equivalent to having the first $\R$-Betti number and first $\Z_2$-Betti number coincide. For an orientable Riemannian surface $S$ of genus $g\geq 1$ of unit area, this gives the following asymptotic estimate:
$$
N_S(t)\geq t^{2g}/ (2g)! +o(t^{2g-1}).
$$
\bigskip

Now we switch our focus to surfaces.

\medskip

The closed geodesics obtained in Corollary \ref{surfaces.Z_2.intro} are also $\Z$-homologically independent, but do not necessarily form a $\Z$-homology basis. Nonetheless, by using techniques specific to dimension $2$, we exhibit $\Z$-homology basis whose length product is uniformly bounded above as follows.

\begin{theorem}\label{surfaces.Z.intro}
Consider an orientable Riemannian surface $S$ of genus $g$.

Then there exists a $\Z$-homology basis induced by closed geodesics $\gamma_1,\ldots,\gamma_{2g}$ whose length product satisfies the following inequality:
$$
\prod_{k=1}^{2g} \ell(\gamma_k)\leq {C^g} \cdot \area(S)^g
$$
for some positive constant $C<2^{31}$.
\end{theorem}

We also prove a similar result for non-orientable surfaces and conjecture the following.

\begin{conjecture}\label{conj.surfaces.Z.intro}
There exists a positive constant $c>0$ such that any Riemannian surface $S$ of first $\Z_2$-Betti number $b>0$ admits a $\Z_2$-homology basis induced by closed geodesics $\gamma_1,\ldots,\gamma_{b}$ whose length product satisfies the following inequality:
$$
\prod_{i=1}^{b} \ell(\gamma_i)\leq (c \, \log b)^{b} \cdot \left({\area(S)\over b}\right)^{b/2}.
$$
\end{conjecture}

There exist infinite sequences of hyperbolic surfaces of growing genus $g$ whose shortest homologically non-trivial closed geodesic (also known as the {\it homological systole}) has length growth rate at least $4/3 \cdot \log g$. The first known construction is due to Buser and Sarnak \cite{BS94}. As a consequence, the minimal length product of a $\Z_2$-homology basis cannot have an upper bound with asymptotic growth in terms of the genus any less than $({4/ 3} \cdot \log g )^{2g}$ for area equal to $4\pi(g-1)$. So Conjecture \ref{conj.surfaces.Z.intro}, if true, would be optimal up to some exponential factor.\\

We now present various partial results supporting this conjecture.\\

First, the analogous result for graphs is true. By graph we mean a finite simplicial complex of dimension $1$ and for which the notion of length $\ell$ of a subgraph (such as a cycle) is simply the number of its edges.
 
 \begin{theorem}
 Consider a connected graph $\Gamma$ of first $\Z_2$-Betti number $b\geq 2$.
 
 Then there exists a $\Z_2$-homology basis induced by cycles $\gamma_1,\ldots,\gamma_b$ whose length product satisfies the following inequality:
 $$
 \prod_{k=1}^{b} \ell(\gamma_k) \leq \left({4e}\log_2 b\right)^b \cdot \left({\ell(\Gamma)\over b}\right)^b.
 $$
\end{theorem}

This result is a straightforward consequence of previous work \cite{BT97, BS02} and an induction argument.\\

Secondly, the conjecture for surfaces with large enough systole can be deduced from \cite[Theorem 1.1]{BPS12}.

\begin{proposition}\label{prod.surfaces.sys.big.intro}
Consider an orientable Riemannian surface $S$ of genus $g$.
Suppose that its homological systole satisfies
$$
\sys(S) \geq \sqrt{\area(S)\over 4\pi (g-1)}.
$$
Then there exists a $\Z_2$-homology basis induced by closed geodesics $\gamma_1,\ldots,\gamma_{2g}$ whose length product satisfies the following inequality:
$$
\prod_{i=1}^{2g} \ell(\gamma_i)\leq \left(C \log g \right)^{2g} \cdot \left({\area(S)\over g}\right)^g
$$
for some positive constant $C<2^{31}$.
\end{proposition}

Next we can ask for the minimal length product over $\Z_2$-homologically independent families of closed geodesics of cardinality less than $2g$. In that direction, we are able to prove the analog of Conjecture \ref{conj.surfaces.Z.intro} up to cardinality $g$.

\begin{proposition}\label{surfaces.gcurves.intro}
Consider an orientable Riemannian surface $S$ of genus $g$. 

Then there exists $\Z_2$-homologically independent closed geodesics $\gamma_1,\ldots,\gamma_g$ whose length product satisfies the following inequality:
$$
\prod_{k=1}^{g}\ell(\gamma_k)\leq \left( C \log g\right)^g \left(\frac{\area(S)}{g}\right)^{g/2}.
$$
for some positive constant $C<2^{31}$.
\end{proposition}

The family of loops appearing in the previous proposition could possibly belong to the same Lagrangian subspace with respect to the intersection form. In the transversal direction and in the particular case of hyperbolic metrics, we are able to prove the following.

\begin{theorem}\label{th.inter.hyp.intro}
There exists a positive constant $c>0$ such that any orientable genus $g\geq 2$ hyperbolic surface $S_0$ has a pair of simple closed geodesics $\gamma$ and $\delta$ pairwise intersecting once and whose length product satisfies the following inequality:
$$
\ell(\gamma) \cdot \ell(\delta)\leq \left(c \, \log (2g)\right)^2.
$$
\end{theorem}

\bigskip

The paper is organized as follows. The next section is dedicated to the proof of Theorem \ref{manifolds.Z_2.intro} and its Corollaries \ref{surfaces.Z_2.intro}, \ref{connectedsum.Z_2.intro} and \ref{nb.geodesics.intro}. In the last section we treat in detail the case of surfaces, mainly focusing on the orientable case. \\

\noindent {\bf Acknowledgements.} We are grateful to A. Abdesselam, I. Babenko, B. Kahn and J. Milnor for valuable exchanges. We are also indebted to J. Gutt whose proof of a symplectic analog of Minkowski's first theorem (see Lemma \ref{LemGut}) inspired a mechanism used in the proof of Theorem \ref{manifolds.Z_2}.\\

\section{Minimal length product over $\Z_2$-homology bases for manifolds} \label{sec:manifolds}%

In this section, we prove Theorem \ref{manifolds.Z_2.intro} and derive Corollaries \ref{surfaces.Z_2.intro}, \ref{connectedsum.Z_2.intro} and \ref{nb.geodesics.intro}.

\subsection{From cup product to length product} %

Let $M$ be a Riemannian manifold of dimension $n$ whose first cohomology group with $\Z_2$-coefficients is non-trivial. Set
$$
b:=\dim H^1(M;\Z_2)>0.
$$
Given a non-zero cohomological class $\alpha \in H^1(M;\Z_2)$, we define its {\it length} by
$$
L(\alpha):=\inf \{\ell(\gamma) \mid \gamma \, \, \text{is a closed curve with} \, \, \alpha[\gamma]\neq 0\}
$$
where $[\gamma]$ denotes the homology class in $H_1(M;\Z_2)$ corresponding to the curve $\gamma$.
The following result was proved in \cite{BK16}.

\begin{theorem}\label{th:cup}
Let $M$ be a Riemannian manifold of dimension $n$ and suppose that there exist (not necessarily distinct) cohomology classes $\alpha_1,\ldots,\alpha_n$ in $H^1(M;\Z_2)$ whose cup product $\alpha_1 \cup \ldots \cup \alpha_n \neq 0 \in H^n(M;\Z_2)$. 
Then
$$
 \prod_{k=1}^n L(\alpha_k)\leq n^n \cdot \vol \, (M).
$$
\end{theorem}

For the reader's convenience, we sketch the proof in the $3$-dimensional case and refer to \cite{BK16} for the general case and further details.

\begin{proof}[Sketch of proof.] The proof is inspired by Guth's proof of isosystolic inequality, see \cite{Gut10}. So we assume in the sequel $n=3$. Furthermore suppose that we have ordered our cohomological classes such that $L(\alpha_1)\leq L(\alpha_2)\leq L(\alpha_3)$.\\

Choose a hypersurface $Z_2$ of $M^3$ with the following properties:
\begin{enumerate}[i)]
\item $Z_2$ is Poincar\'e dual to $\alpha_3$. In particular:
\begin{itemize}[label=$\cdot$,itemsep=1ex,leftmargin=0.2cm]
\item for any $z \in H_1(M;\Z_2)$, $\alpha_3(z)=z \cap [Z_2]$ where $\cdot \cap \cdot$ denotes the intersection number modulo $2$ between two cycles of complementary dimension;
\item for any $\omega \in H^2(M;\Z_2)$, $\omega [Z_2]=\omega \cup \alpha_3 [M]$;
\end{itemize}
\item $Z_2$ is (almost) minimal\footnote{The notion of almost minimality is that $Z_2$ is minimal for area among hypersurfaces in its cohomological class up to some choosen small additive constant. Because this constant can be fixed as small as wanted, the proof remains the same if we actually suppose that $Z_2$ is minimal. We do so and refer the reader to the original proof in \cite{BK16} for more details.
} for the area among hypersurfaces in the class $[Z_2]$.
\end{enumerate}

Now observe that the restrictions of $\alpha_1,\alpha_2$ to $Z_2$ (denoted by $\alpha'_1,\alpha'_2$) satisfy
\begin{itemize}[label=$\cdot$,itemsep=1ex] 
\item $\alpha'_1 \cup \alpha'_2 [Z_2] = \alpha_1 \cup \alpha_2 \cup \alpha_3 [M] =1$;
\item $L(\alpha'_i)\geq L(\alpha_i)$ for $i=1,2$.
\end{itemize}

Define $Z_1$ as a minimal closed geodesic of $Z_2$ Poincar\'e dual to $\alpha'_2$. 
In particular $\alpha_2(z)=z \cap [Z_1]$ for any $z \in H_1(Z_2;\Z_2)$.\\

Pick some point $z_0 \in Z_1$ and define the following nested sets:
$$
\begin{array}{c}
D_1:= \left\{z \in Z_1 \mid d(z,z_0)\leq R_1\right\} \\
\bigcap \\
D_2:=\left\{z \in Z_2 \mid d(z,D_1)\leq R_2\right\} \\
\bigcap \\
D_3:=\left\{z \in M \mid d(z,D_2)\leq R_3\right\}
\end{array}
$$
with $R_i:={L(\alpha_i) / 6}$ for $i=1,2,3$. Obviously $\ell(D_1)=2R_1$.\\

By the coarea formula we have 
$$
\area(D_2)=\int_0^{R_2} \ell(\partial D_2(r)) dr
$$
where $D_2(r):=\{z \in Z_2 \mid d(z,D_1)\leq r\}$.\\

But $Z_1\cap D_2(r)=0 \in H_1(D_2(r),\partial D_2(r); \Z_2)$ for $r \in (0,R_2)$. 
If not, by Lefschetz-Poincar\'e duality, there exists a closed geodesic $\gamma$ with non-zero homology class in $H_1(D_2(r);\Z_2)$ such that $[\gamma] \cap Z_1=\alpha_2[\gamma] =1$. Observe that by construction $d(z,z_0)< R_1+R_2 \leq L(\alpha_2)/2$ for any $z \in D_2(r)$.
According to Gromov's curve factoring lemma (see \cite{Gut10}), the loop $\gamma$ decomposes into a sum $\sum_i \gamma_i$ of loops of length less than $L(\alpha_2)$. But for some index $i_0$ we must have $\alpha_2[\gamma_{i_0}]=1$ which gives a contradiction. 
Thus $\ell(\partial D_2(r))\geq 2 \ell(Z_1 \cap D_2(r))\geq 4R_1$ for any $r \in (0,R_2)$ by minimality of $Z_1$ from which we derive by integration that $\area(D_2)\geq 4R_1R_2$.\\

In exactly the same way we have 
$$
\vol(D_3)=\int_0^{R_3} \area(\partial D_3(r)) dr
$$
with $D_3(r):=\{z \in M \mid d(z,D_2)\leq r\}$. Because $Z_2\cap D_3(r)=0 \in H_2(D_3(r),\partial D_3(r); \Z_2)$ for $r \in (0,R_3)$, the minimality of $Z_2$ implies the lower bound $$\area(\partial D_3(r))\geq 2 \area(Z_2 \cap D_3(r))\geq 8R_1R_2$$
from which we derive by integration that $\vol(D_3)\geq 8R_1R_2R_3$.\\

This implies the desired inequality
$$
L(\alpha_1) \cdot L(\alpha_2) \cdot L(\alpha_3)\leq 3^3 \cdot \vol \, (M).
$$
\end{proof}

To fully exploit this result, we introduce the following map. As the dimension $b$ of the first cohomology group $H^1(M;\Z_2)$ is positive, the $\Z_2$-fundamental class of $M$ induces a map
\begin{eqnarray*}
F_M: H^1(M;\Z_2) \times \ldots \times H^1(M;\Z_2) & \to & \Z_2\\
(\beta_1,\ldots,\beta_n)&\mapsto & \beta_1\cup \ldots \cup \beta_n [M].
\end{eqnarray*}
This map is $n$-multilinear and symmetric.\\

\subsection{Algebraic invariants of multilinear forms over $\Z_2$-vector spaces}\label{sec:alginv} %

We now introduce the class of invariants of multilinear maps which will be relevant to our purpose.\\

Let $V$ be a $\Z_2$-vector space of dimension $b$ and $n$ a positive integer. We denote by ${\mathcal F}_n(V)$ the set of $n$-multilinear forms on $V$. If we fix a basis ${\mathcal B}:=(e_1,\ldots,e_b)$ of $V$, an element $F \in {\mathcal F}_n(V)$ decomposes as
$$
F=\sum_{(i_1,\ldots,i_n) \in \llbracket 1,b \rrbracket^n} F_{i_1,\ldots,i_n} \, e_{i_1}^\ast\otimes \ldotsÊ\otimes e_{i_n}^\ast
$$
where $F_{i_1,\ldots,i_n}:=F(e_{i_1},\ldots,e_{i_n})\in \Z_2$. The $n$-dimensional matrix $\hat{A}_{\mathcal B}(F):=(F_{i_1,\ldots,i_n})$ with $\Z_2$-entries thus encodes the multilinear map $F$.

\begin{definition}
An algebraic invariant over ${\mathcal F}_n(V)$ is a map 
\begin{eqnarray*}
\kappa: {\mathcal F}_n(V) & \to & \Z_2\\
F&\mapsto & \kappa(F)
\end{eqnarray*}
which coincides for some homogeneous polynomial with its evaluation on the entries of the matrix $\hat{A}_{\mathcal B}(F)$ for any basis ${\mathcal B}$ of $V$.
\end{definition}

An example of such an algebraic invariant over ${\mathcal F}_2(V)$ is given by the reduction modulo $2$ of the determinant. More precisely, the expression defined by
$$
\sum_{\sigma \in {\mathcal S}_n} \prod_{k=1}^b F_{k,\sigma(k)}
$$
does not depend on the basis in which we have decomposed $F$. This defines an algebraic invariant we will denote by $\Det_2$. If $F$ is symmetric, Milnor \cite{Mil70} remarks that the term corresponding to $\sigma$ cancels the term corresponding to $\sigma^{-1}$ unless $\sigma=\sigma^{-1}$. Therefore for symmetric bilinear forms
we can write
$$
\Det_2 F=\sum_{P_1\cup\ldots \cup P_q =\{1,\ldots,b\}} \, \, \prod_{j=1}^q F(P_j)
$$
where the sum is taken over all partitions of the set $\{1,\ldots,b\}$ into one or two elements subsets $P_j$ with the convention that $F(\{i\})=F_{ii}$ and $F(\{i,j\})=F_{ij}$. This remark will be helpful in the sequel.\\

This example generalizes to higher dimensions using the {\it hyperdeterminant} as introduced by Cayley \cite{Cay45} and later studied by Gelfand, Kapranov and Zelevinsky (see \cite{GKZ92} and the references therein).

More precisely, consider the Segre embedding of the product $X={\mathbb P}^{b-1}\times \ldots \times {\mathbb P}^{b-1}$ of $n$ complex projective spaces of dimension ${b-1}$ into the projective space ${\mathbb P}^{b^n-1}$. According to \cite[Theorem 1.3]{GKZ92}, the projective dual variety $X^\vee$ consisting of all hyperplanes in ${\mathbb P}^{b^n-1}$ tangent to $X$ at some point is a hypersurface in the dual projective space $\left({\mathbb P}^{b^n-1}\right)^\ast$. In particular, there exists a (unique up to sign) homogeneous polynomial with integer coefficients and irreducible over $\Z$ which defines $X^\vee$. This homogeneous polynomial is called {\it hyperdeterminant} and denoted by $\Det$. The hyperdeterminant is thus a homogeneous polynomial function on the space $({\mathbb C}^b)^\ast \otimes \ldots \otimes ({\mathbb C}^b)^\ast$ of $n$-multilinear maps over ${\mathbb C}^b$. If ${\mathbb C}^b$ is equipped by a basis ${\mathcal C}=(v_1,\ldots,v_b)$ an element $f \in ({\mathbb C}^b)^\ast \otimes \ldots \otimes ({\mathbb C}^b)^\ast$ is represented by a $n$-dimensional complex matrix $A_{\mathcal B}(f)=(a_{i_1,\ldots,i_n})$ where $a_{i_1,\ldots,i_n}=f(e_{i_1},\ldots,e_{i_n})$, and so $\Det(f)=\Det(A_{\mathcal B}(f))$ is a homogeneous polynomial function of matrix entries. An important property of the hyperdeterminant is its relative invariance under the action of the group $GL({\mathbb C}^b)\times \ldots \times GL({\mathbb C}^b)$, and thus its invariance by $SL({\mathbb C}^b)\times \ldots \times SL({\mathbb C}^b)$, according to \cite[Proposition 1.4]{GKZ92}. In particular $\Det A_{\mathcal{B}}(f)$ remains unchanged up to sign if we commute two elements $v_i$ and $v_j$, or if we replace $v_i$ by $v_i+v_j$ in the basis ${\mathcal C}$.

Now fix an element $F \in {\mathcal F}_n(V)$. Given a basis ${\mathcal B}:=(e_1,\ldots,e_b)$ of $V$, any $n$-multilinear form $F$ over $V$ is represented by an unique $n$-dimensional {\it real} matrix
$$
A_{\mathcal{B}}(F) =(a_{i_1,\ldots,i_n})_{(i_1,\ldots,i_n) \in \llbracket 1,b \rrbracket^n}
$$
where each coefficient is either $0$ or $1$ according to the equation $F(e_{i_1},\ldots,e_{i_n})=a_{i_1,\ldots,i_n}\mod \, 2$. As $\Det$ is a homogeneous polynomial function with integers coefficient, the expression $\Det A_{\mathcal{B}}(F)$ belongs to $\Z$.
Its invariance under the action of the group $SL({\mathbb C}^b)\times \ldots \times SL({\mathbb C}^b)$ implies that 
$$
\Det A_{\mathcal{B}}(F)=\Det A_{\mathcal{B}'}(F) \, \, \, \text{mod} \, \, \, 2
$$
if ${\mathcal B}'$ is obtained from ${\mathcal B}$ by commuting two elements $e_i$ and $e_j$ or by replacing $e_i$ by $e_i+e_j$. Because these transformations span the linear group $GL(V)$ we can set
$$
\Det_2 F:=\Det A_{\mathcal{B}}(F) \, \, \, \text{mod} \, \, \, 2
$$
which does not depend on the chosen basis ${\mathcal B}$. This reduction modulo $2$ of the hyperdeterminant $\Det_2 : {\mathcal F}_n(V) \to \Z_2$ is thus an algebraic invariant as obviously $\hat{A}_{\mathcal B}(F)=A_{\mathcal{B}}(F) \, \, \, \text{mod} \, \, \, 2$.\\

In dimension $n=2$, the hyperdeterminant coincides with the usual determinant so the algebraic invariant $\Det_2$ coincides with the one previously defined on ${\mathcal F}_2(V)$.\\

In dimensions $n=3$ and $b=2$, the hyperdeterminant has an explicit formula first discovered by Cayley \cite[p.89]{Cay45} and whose reduction modulo $2$ gives the following expression:
$$
\Det_2 F=F_{111}\cdot F_{222}+F_{122}\cdot F_{211}+F_{212}\cdot F_{121}+F_{221}\cdot F_{112}. 
$$
In particular, we deduce that $\Det_2 F=F_{111}\cdot F_{222}+F_{122}\cdot F_{211}$ in the case where $F$ is symmetric.\\

We now introduce the following notion.

\begin{definition}
An algebraic invariant $\kappa$ over ${\mathcal F}_n(V)$ is said {\it balanced} if 
\begin{itemize}[label=$\cdot$,itemsep=1ex]
\item there exist $q,m \in \N^\ast$ such that $qn=mb$,
\item the homogeneous polynomial function associated to $\kappa$ can be written as
\begin{equation}\label{eq:bal}
\kappa(F)=\sum_{I} \prod_{(i_1,\ldots,i_n) \in I} F_{i_1,\ldots,i_n},
\end{equation}
where each $I$ can be described as the subset of families of indices of length $n$
$$
I=\{(\pi(1),\ldots,\pi(n)),(\pi(n+1),\ldots,\pi(2n)),\ldots,(\pi((q-1)n+1),\ldots,\pi(qn))\}.
$$
for some map 
$$
\pi : \{1,\ldots,qn\} \to \{1,\ldots,b\}
$$
satisfying $|\pi^{-1}(i)|=m$ for $i=1,\ldots,b$.
\end{itemize}
\end{definition}

In other terms, we require the $I$s in the sum above to be obtained by taking $m$ times the indices $\{1,\ldots,b\}$ and rearranging them into $q$ families of length $n$.\\

For example, it is easy to check from its expression that $\Det_2$ is balanced in the case $n=2$, as well as in the case when $n=3$ and $b=2$. This fact indeed generalizes as follows.

\begin{proposition}
The algebraic invariant $\Det_2$ is balanced for any $n$ and $b$.
\end{proposition}

\begin{proof}
This is a straightforward consequence of the main result \cite{Cra27} which proves, using tensorial algebra, that any algebraic relative invariant over the space of multilinear forms of a real vector space has this property.
\end{proof}
Note that we did not find examples of {\it unbalanced} algebraic invariants in the litterature, but whether or not they exist is possibly well known to experts.\\

\subsection{From a non-zero balanced algebraic invariant to a length product inequality} %

We now prove the main result of this section.

\begin{theorem}\label{manifolds.Z_2}
Consider a Riemannian manifold $M$ of first Betti number $b>0$ and of dimension $n$. Suppose that 
$$
\kappa(F_M)=1
$$
for some balanced algebraic invariant $\kappa$. 

Then there exists a $\Z_2$-homology basis induced by closed geodesics $\gamma_1,\ldots,\gamma_{b}$ whose length product satisfies the following inequality:
$$
\prod_{k=1}^{b} \ell(\gamma_k)\leq {n^b} \cdot \vol(M)^{b/ n}.
$$
\end{theorem}

Theorem \ref{manifolds.Z_2.intro} is the direct application of this result with $\kappa=\Det_2$.

\begin{proof}
 Choose a sequence of $\Z_2$-homologically independent closed geodesics $\gamma_1,\ldots,\gamma_b$ in $M$ corresponding to the {\it successive minima} in the sense that
$$
\ell(\gamma_k)=\min \{ L \mid \text{there exist} \, \, k \, \, \text{closed curves} \, \, \Z_2\text{-homologically independent of length at most} \, \, L\}
$$
for $k=1,\ldots,b$. The corresponding family of homology classes $([\gamma_1],\ldots,[\gamma_b])$ is a basis of $H_1(M;\Z_2)$ and we denote by ${\mathcal B}:=(\alpha_1,\ldots,\alpha_b)$ the dual basis in $H^1(M;\Z_2)$ defined by $\alpha_i[\gamma_j]=\delta_{ij}$. By construction,
$$
L(\alpha_k)=\ell(\gamma_k)
$$
for $k=1,\ldots,b$.

Now consider a balanced homogeneous polynomial function associated to the algebraic invariant $\kappa$ as in formula (\ref{eq:bal}). As $\kappa(F_M)=1$, we have that
$$
\prod_{j=1}^{q} F_M(\alpha_{\pi((j-1)n+1)},\ldots,\alpha_{\pi(jn)})=1
$$
for some map 
$$
\pi : \{1,\ldots,qn\} \to \{1,\ldots,b\}
$$
satisfying $|\pi^{-1}(i)|=m$. In particular for all $j=1,\ldots,q$
$$
\alpha_{\pi((j-1)n+1)} \cup \ldots \cup \alpha_{\pi(jn)}[M]=F_M(\alpha_{\pi((j-1)n+1)},\ldots,\alpha_{\pi(jn)})=1
$$
which implies according to Theorem \ref{th:cup} that
$$
L(\alpha_{\pi((j-1)n+1)})\ldots L(\alpha_{\pi(jn)})\leq n^n \cdot \vol \, (M).
$$
We deduce that
$$
\left(\prod_{k=1}^b L(\alpha_k)\right)^m=\prod_{j=1}^q L(\alpha_{\pi((j-1)n+1)})\ldots L(\alpha_{\pi(jn)})\leq n^{qn} \cdot \vol \, (M)^q.
$$
As such
$$
\prod_{k=1}^b L(\alpha_k)\leq n^{qn/m} \cdot \vol \, (M)^{q/m}
$$
and we can conclude using the relations $qn=mb$ and $\ell(\gamma_k)=L(\alpha_k)$ for all $k=1,\ldots,b$.
\end{proof}

\smallskip

\subsection{Minimal length product over $\Z_2$-homology bases for surfaces} %

In the case where $M$ is an orientable Riemannian surface of genus $g$, the map $F_M$ is alternating. Indeed if $a \in H_1(M;\Z_2)$ is the Poincar\'e dual of a class $\alpha \in H^1(M;\Z_2)$, we have
$$
F_M(\alpha,\alpha)=\alpha \cup \alpha [M]=\alpha(a)=a\cap a=0.
$$
This implies that in any basis of $H^1(M;\Z_2)$ we can write
$$
\Det_2 F_M=\sum F_M(P_1)\ldots F_M(P_g)
$$
where the sum is taken over partitions of $\{1,\ldots,2g\}$ into two elements subsets. This is exactly the reduction of the Pfaffian modulo $2$ which is consistent with the fact that $x^2=x$ in characteristic $2$. It is well known that we can choose a basis $(a_1,b_1,\ldots,a_g,b_g)$ of $H_1(M;\Z_2)$ such that 
$$
a_i \cap b_i=1
$$
for all $i=1,\ldots,g$, the other intersection numbers between two elements of the basis being zero. If we denote by $(\alpha_1,\beta_1,\ldots,\alpha_g,\beta_g)$ the basis of $H^1(M;\Z_2)$ such that $\alpha_i$ (resp. $\beta_i$) is the Poincar\'e dual class of $a_i$ (resp. $b_i$), then we get that $\alpha_i\cup \beta_i [M]=a_i \cap b_i=1$ while the other cup products are zero. This implies that there exists only one partition for which the term $F_M(P_1)\ldots F_M(P_g)$ is non zero and thus that
$$
\Det_2 F_M=1. 
$$
This proves Corollary \ref{surfaces.Z_2.intro} when the surface is orientable.\\

In the case where $M$ is a non-orientable Riemannian surface we can argue similarly. First recall that in any basis of $H^1(M;\Z_2)$ we can write
$$
\Det_2 F_M=\sum_{P_1\cup \ldots \cup P_q=\{1,\ldots,b\} \, \,} \prod_{j=1}^q F_M(P_j)
$$
where the sum is taken over partitions of $\{1,\ldots,2g\}$ into one or two elements subsets. 

Suppose that the first Betti number is odd and set $b=2g+1$ (possibly $g$ can be $0$). We can always find a basis $(a_1,b_1,\ldots,a_g,b_g,c)$ of $H_1(M;\Z_2)$ such that $a_i \cap b_i = c\cap c=1$ for all $i=1,\ldots,g$ while the other intersection numbers are zero. 
If we denote by $(\alpha_1,\beta_1,\ldots,\alpha_g,\beta_g,\zeta)$ the basis of $H^1(M;\Z_2)$ obtained by considering the Poincar\'e dual classes of $(a_1,b_1,\ldots,a_g,b_g,c)$, we get that $\alpha_i\cup \beta_i [M]=\zeta \cup \zeta [M]=1$, while the other cup products are zero. It implies that there exists only one partition for which the term $F_M(P_1)\ldots F_M(P_q)$ is not zero and thus
$$
\Det_2 F_M=1.
$$
If the first Betti number is even, we can find a basis $(a_1,b_1,\ldots,a_g,b_g,c_1,c_2)$ of $H_1(M;\Z_2)$ such that $a_i \cap b_i = c_j\cap c_j=1$ and the other intersection numbers are zero. In a similar way we can check that $\Det_2 F_M=1$.

This proves the non-orientable case of Corollary \ref{surfaces.Z_2.intro}.

\subsection{Minimal length product over $\Z_2$-homology bases for $\R P^3 \# \R P^3$} %

For $M=\R P^3 \# \R P^3$ we have 
$$
\Det_2 F_M=(F_M)_{111}\cdot (F_M)_{222}+(F_M)_{122}\cdot (F_M)_{211}
$$
as $F_M$ is symmetric, and the dimension $n$ of the manifold and the dimension $b$ of its first homology group are respectively $n=3$ and $b=2$.

Now consider a decomposition of $M$ into two copies of $(\R P^3\setminus B^3)_i$ for $i=1,2$ glued along their boundaries via a diffeomorphism. According to the Mayer-Vietoris sequence $H_1(M;\Z_2)\simeq H_1(\R P^3;\Z_2)\oplus H_1(\R P^3;\Z_2)$ where the generator of the $i$-th component in this direct composition can be realized by a closed curve $a_i \subset (\R P^3\setminus B^3)_i$ for $i=1,2$. If we denote by ${\mathcal B}=(\alpha_1,\alpha_2)$ the basis of $H^1(M;\Z_2)$ formed by considering the Poincar\'e dual class of $(a_1,a_2)$ in $M$, a straightforward calculation shows that $(F_M)_{111}=(F_M)_{222}=1$, and the other terms are all zero. Thus
$$
\Det_2 F_M=1.
$$
This proves Corollary \ref{connectedsum.Z_2.intro}.\\


\subsection{Homological asymptotic volume and asymptotic number of closed geodesics}\label{sec:stab}


The Riemannian universal abelian covering space $\bar{M}$ of a Riemannian manifold $M$ is the normal covering space associated to the commutator subgroup $[\pi_1M,\pi_1M]$ of the fundamental group of $M$ endowed with the Riemannian pullback metric. 
If the first integral homology group of $M$ has free rank $d>0$, it is well known that the limit
$$
\Omega_{ab}(M):=\lim_{R \to \infty} \frac{\vol\, B(\bar{x},R)}{R^d}
$$
exists, is positive and does not depend on the choice of the base point $\bar{x} \in \bar{M}$. This Riemannian invariant appears in \cite{Bab92} under the name of {\it homological asymptotic volume} and is related to the stable norm as follows. First recall that the stable norm $\|\cdot \|$ is defined on $H_1(M,\R)$ by 
$$
\|h \|:= \inf \, \{\sum^{n}_{i=1}|r_i|\cdot \ell (\gamma_i) \}
$$
where the infimum is taken over all real Lipschitz $1$-cycles $\sum_{i=1}^n r_i \gamma_i$ realising a given class $h$. If we denote by ${\mathcal B}(M)$ the unit ball of the stable norm, then
\begin{equation}\label{eq:stab}
\Omega_{ab}(M)=\mu \left({\mathcal B}( M)\right) \cdot \vol(M)
\end{equation}
where $\mu$ denotes the Haar measure on $H_1(M,\R)\simeq \R^d$ for which the fundamental parallelepiped of the lattice $i \, H_1(M,\Z)$ has unit volume. Here $i : H_1(M,\Z) \to H_1(M,\R)$ denotes the homology mapping induced by the inclusion $i : \Z \hookrightarrow \R$ of coefficients.
The volume of the unit stable ball turns out to be related to length products through the following.

\begin{proposition}\label{stab}
Let $(\gamma_1,\ldots,\gamma_{d})$ be a family of closed geodesics inducing a basis of $H_1(M,\R)$. 
Then
$$
\mu ({\mathcal B}( M))\cdot \prod_{k=1}^{d} \ell(\gamma_k)\geq {2^d \over d!}.
$$
\end{proposition}

This inequality still holds true for finite simplicial complexes endowed with Riemannian polyhedral metrics. In such a generality this inequality then turns out to be optimal, the equality being reached for the wedge product of $d$ circles (compare with \cite{BB06}).

\begin{proof}
First observe that $\text{Conv}\left\{\pm{[\gamma_1] / \|[\gamma_1]\|},\ldots,\pm{[\gamma_d] / \|[\gamma_d]\|}\right\} \subset {\mathcal B}(M)$ where Conv denotes the convex hull of a finite collection of points. Because for any closed geodesic $\gamma$ 
$$
\ell(\gamma)\geq \|[\gamma]\|,
$$
this implies that
\begin{eqnarray*}
\mu({\mathcal B}(M)) \cdot \prod_{k=1}^{d} \ell(\gamma_k) & \geq & \mu({\mathcal B}(M)) \cdot \prod_{k=1}^{d} \|[\gamma_k]\| \\
 & \geq &\mu (\text{Conv} \{\pm[\gamma_1],\ldots,\pm [\gamma_d]\})\\
 &\geq & {2^d \over d!}.
\end{eqnarray*}
The last inequality is a consequence of the observation that $\Z [\gamma_1] \oplus \ldots \oplus \Z [\gamma_d]$ is a sublattice of $i\left(H_1(M,\Z)\right)$.
\end{proof}

Using this proposition we can now bound from below the homological asymptotic volume as follows. 

\begin{theorem}\label{th.stab}
Consider a Riemannian manifold $M$ of first $\Z_2$-Betti number $b>0$ and of dimension $n$. Suppose that the integral homology has no $2$-torsion and that 
$$
\Det_2 F_M \neq 0.
$$
Then 
$$
\Omega_{ab}(M)\geq {\left({2 \over n}\right)^b} \cdot {\vol(M)^{1-b/n} \over b!}.
$$
\end{theorem}

\begin{proof}
Remember that for any abelian group $G$ 
$$
H_1(M,G) \simeq H_1(M,\Z) \otimes G
$$ 
according to the so-called ``universal coefficient Theorem" (see \cite[Theorem 3A.3 and Proposition 3A.5]{Hatcher}. This implies that if $H_1(M,\Z)$ has no $2$-torsion then the rank of $H_1(M,\Z_2)$ and $H_1(M,\R)$ coincide.

By applying Theorem \ref{manifolds.Z_2.intro} we find a $\Z_2$-homology basis of $M$ made of closed geodesics $(\gamma_1,\ldots,\gamma_b)$ such that
$$
\prod_{k=1}^{b} \ell(\gamma_k)\leq {n^b} \cdot \vol(M)^{b/n}.
$$
Because the family of closed geodesics $(\gamma_1,\ldots,\gamma_b)$ induces a basis of $H_1(M,\R)$, Proposition \ref{stab} implies that
$$
\mu({\mathcal B}( M)) \geq {\left({2 \over n}\right)^b} \cdot {\vol(M)^{-b/n} \over b!}
$$
and we can conclude using equality (\ref{eq:stab}).
\end{proof}

In particular we get the following lower bound for surfaces.

\begin{corollary}\label{stab.surfaces}
Consider an orientable Riemannian surface $S$ of genus $g$. 
Then
$$
\Omega_{ab}(S)\cdot \area(S)^{g-1} \geq {1 \over (2g)!}.
$$
\end{corollary}

This improves the previous known lower bound due to Babenko \cite[Corollary 5.5]{Bab92}.\\

To conclude this section let us denote by $N_M(t)$ the number of $\R$-homologically distinct closed geodesics of length $\leq t$. This number has polynomial growth of order the first $\R$-Betti number $b'$ which is captured by the asymptotic volume of the covering space corresponding to the natural map $\pi_1 M \to H_1(M;\Z)/{\mathrm{Tors}} \simeq i\left(H_1(M;\Z)\right)$. Because this asymptotic volume coincides with the homological asymptotic volume and because the stable norm is the unique norm enclosing the asymptotic geometry of this covering space (see \cite{Bur92} and \cite{CS16} for details), one obtains the equality
$$
\lim_{t \to \infty}{N_M(t) \over t^{b'}}=\mu({\mathcal B}(M)).
$$
From this we can deduce the following.

\begin{corollary}\label{nb.geodesics}
Consider a Riemannian manifold $M$ satisfying the hypothesis of Theorem \ref{th.stab}. Then 
$$
N_M(t)\geq {2^b \over \vol(M)^{b/n} \cdot n^b \cdot b!}\, t^b +o(t^{b-1}).
$$
\end{corollary}

A similar lower bound can be deduced for the number of $\R$-homologically distinct and non-multiple closed geodesics of length $\leq t$, compare with \cite[Theorem 7.1 and Corollary 7.2]{Bab92}.


\section{Minimal length product of closed geodesics for surfaces} %


In this section, after a discussion about graphs, we focus on surfaces.

\subsection{Minimal length product over $\Z_2$-homology bases for graphs} %

In this subsection we show that a Minkowski's second principle also holds for graphs. This result is a straightforward consequence of Bollob\'as-Szemer\'edi-Thomason's systolic inequality and paves the way for our conjecture on length product inequalities for surfaces. \\

Consider a finite and connected $1$-dimensional simplicial complex $\Gamma$ endowed with a metric. Here by metric we mean a function which associates to every edge $e$ a positive real value $\ell(e)$ encoding its length. In short, we will call such a pair $(\Gamma,\ell)$ a {\it connected metric graph}. We denote by $E(\Gamma)$ its set of edges and $V(\Gamma)$ its set of vertices. We introduce several quantities. 
First its total length
$$
\ell(\Gamma)=\sum_{e \in E(\Gamma)} \ell(e),
$$
which can be interpreted as its $1$-dimensional volume. 
Secondly, we associate to any path $P=(e_1,\ldots,e_n)$ made of consecutive edges $e_1,\ldots,e_n \in E(\Gamma)$ its length $\ell(P)$ defined by 
$$
\ell(P)=\sum_{i=1}^n \ell(e_i).
$$
A cycle will mean a closed embedded path. The first Betti number of the graph (sometimes called its {\it cyclic number}) can be computed using the formula
$$
b_1(\Gamma)=|E(\Gamma)|-|V(\Gamma)|+1,
$$ 
where $|\cdot |$ denotes cardinality of finite sets.

Minkowski's second principle for graphs takes the following form.

\begin{theorem}\label{prod.graph}
Consider a connected metric graph $(\Gamma,\ell)$ with first Betti number $b\geq 2$.

Then there exists a $\Z_2$-homology basis induced by cycles $\gamma_1,\ldots,\gamma_b$ whose length product satisfies the following inequality:
\begin{equation}\label{eq:prodgraphs}
\prod_{k=1}^{b} \ell(\gamma_k)\leq {4^{b-1}\over (b-1)!} \cdot \prod_{k=2}^{b} \log_2 k \cdot \ell(\Gamma)^b.
\end{equation}
In particular
$$
\left(\prod_{k=1}^{b} \ell(\gamma_k)\right)^{1/b}\leq {4 \over \sqrt[b]{b!}} \cdot \log_2 b \cdot \ell(\Gamma)\underset{b \to \infty}{\simeq} 4e \cdot \log_2 b \cdot {\ell(\Gamma) \over b}.
$$
\end{theorem}

\begin{proof}
Bollob\'as-Szemer\'edi-Thomason's systolic inequality for metric graphs (see~\cite{BT97,BS02}), says that there exists a non-trivial cycle $\gamma_1 \in H_1(\Gamma;\Z_2)$ such that 
$$
\ell(\gamma_1)\leq 4 \, \frac{\log_2 b}{b-1} \cdot \ell(\Gamma).
$$
Now by deleting one edge of $\gamma_1$ we get a connected subgraph $\Gamma_1$ with $b_1(\Gamma_1)=b-1$ and $\ell(\Gamma_1)\leq \ell(\Gamma)$. 

If $b=2$, we choose the only remaining minimal cycle $\gamma_2 \subset \Gamma_1$ whose length satisfies 
$$
\ell(\gamma_2)\leq \ell(\Gamma_1)\leq \ell(\Gamma).
$$
Then we are done as the cycles $\gamma_1$ and $\gamma_2$ induce a $\Z_2$-homology basis by construction, and their length product satisfies the desired upper bound.

If $b>2$ we argue by induction: because $\Gamma_1$ together with the metric induced by $\ell$ is also a connected metric graph, there exists a family of cycles $(\gamma_2,\ldots,\gamma_b)$ which form a homology basis of $H_1(\Gamma_1;\Z_2)$ and whose length product satisfies
$$
\prod_{k=2}^{b} \ell(\gamma_k)\leq {4^{b-2}\over (b-2)!} \cdot \prod_{k=2}^{b-1} \log_2 k \cdot \ell(\Gamma_1)^{b-1}.
$$
In turn, the cycles $\gamma_1,\ldots,\gamma_b$ form a $\Z_2$-homology basis of $\Gamma$ by construction. Furthermore their length product satisfies the desired bound.
\end{proof}

This inequality cannot be substantially improved as one can construct sequences of $p$-regular graphs whose girth---defined as the shortest length of a non-trivial cycle---grows logarithmically in the number of vertices. These examples are metric graphs with all edges of length $1$. The best known constructions have asymptotic girth
$$
{4 \over 3} \log_{p-1} n
$$
as the number of vertices $n$ goes to infinity (see \cite{Wei84} and \cite{LPS88}). The $3$-regular case of these examples provides a sequence of graphs whose first Betti number $b \to \infty$ and for which every homology basis $(\gamma_1,\ldots,\gamma_b)$ satisfies the asymptotic lower bound
$$
\left(\prod_{k=1}^{b} \ell(\gamma_k)\right)^{1/b}\gtrsim {4 \over 9} \cdot \log_2 b \cdot {\ell(\Gamma) \over b}.
$$

By analogy we conjecture the following for surfaces.

\begin{conjecture}\label{conj.surfaces.Z}
There exists a positive constant $c>0$ such that the following holds. Consider a Riemannian surface $S$ of first $\Z_2$-Betti number $b>0$. Then there exists a $\Z_2$-homology basis induced by closed geodesics $\gamma_1,\ldots,\gamma_{b}$ and whose length product satisfies the following inequality:
$$
\prod_{i=1}^{b} \ell(\gamma_i)\leq (c \, \log b)^{b} \cdot \left({\area(S)\over b}\right)^{b/2}.
$$
\end{conjecture}


\subsection{Minimal length product over $\Z$-homology bases for orientable surfaces}%

Observe that Theorem \ref{prod.graph} and its proof are still valid for $\Z$-coefficients.\\

On an orientable surface $S$ we could try to use the intersection form on $H_1(S,\Z)$ to deduce an analog of Corollary \ref{surfaces.Z.intro} for $\Z$-coefficients. Indeed, the closed geodesics composing a minimal $\Z_2$-homology basis are {\it straight}, by which we mean that between any two points on the curve the distance equals the length of the shortest segment of the curve joining them. This implies that any two closed geodesics in a minimal $\Z_2$-homology basis (for successive lengths) have pairwise {\it geometric} intersection $0$ or $1$. Combining this with the fact that they are linearly independent in $H_1(S,\Z)$, one could hope to deduce that they indeed form a $\Z$-homology basis but at the moment, it is unclear to us how to make this strategy work.\\

Nonetheless, we prove using a different strategy the following Minkowski's second principle for orientable Riemannian surfaces and in any event, this different approach will also be useful in the sequel. 
\begin{theorem}\label{prod.Z.surfaces}
Consider an orientable Riemannian surface $S$ of genus $g$.

Then there exists a $\Z$-homology basis induced by closed geodesics $\gamma_1,\ldots,\gamma_{2g}$ whose length product satisfies the following inequality:
\begin{equation*}
\prod_{i=k}^{2g} \ell(\gamma_k)\leq {C^g} \cdot \area(S)^g
\end{equation*}
for some constant $C<2^{31}$.
\end{theorem}

\begin{remark}\label{rem:prodtorus}
This result for genus $g=1$ is already well-known. More precisely, in \cite{CK03} the following optimal inequality was proved: on any Riemannian 2-torus of unit area, we can always find two closed geodesics $\gamma_1,\gamma_2$ which generate the fundamental group and whose length product satisfies
$$
\ell(\gamma_1) \cdot \ell(\gamma_2) \leq {2 \over \sqrt{3}}.
$$
Furthermore there is equality if and only if the torus is isometric to a flat hexagonal torus (the torus obtained by pairing the opposite sides of a flat regular hexagon).
\end{remark}

\begin{proof}
The homological systole is defined by 
$$
\sys(S):=\min \{\ell(\gamma) \mid [\gamma]\neq 0 \in H_1(S;\Z)\}
$$
where the minimum is taken over loops whose homological class is non trivial. A closed geodesic $\gamma$ is said to realize the homological systole if $[\gamma]\neq 0$ and
$$
\ell(\gamma)=\sys(S).
$$
Let $\cdot \cap \cdot$ denote the intersection form on $H_1(S;\Z)$ (this notation will be used throughout the sequel). If $\gamma$ is a closed geodesic with $[\gamma]\neq0$ we can now define
$$
L(\gamma):=\min \{\ell(\delta) \mid [\gamma]\cap [\delta]\neq 0\}.
$$
Observe that $L(\gamma)$ depends only on the homological class induced by $\gamma$ and is realized by the length of some closed geodesic $\delta$ with non zero intersection with $\gamma$. Furthermore if $\gamma$ is the homological systole, then it is easy to see by a standard surgery argument that in this case $[\gamma]\cap [\delta]=1$.

\begin{lemma}\label{lemma.inter.systole}
If a closed geodesic $\gamma$ realizes the homological systole, then
$$
\ell(\gamma) \cdot L(\gamma) \leq 2 \area(S).
$$
\end{lemma} 

This is a sort of dual version of Theorem \ref{th:cup} for $\Z$-coefficients and specialized to the systole.

\begin{proof}[Proof of the lemma]
Pick a point $p$ on $\gamma$ and denote by $D(p,R)$ the set of points at a distance at most $R$ from $p$. For any $R \in (0,L(\gamma)/2)$ the curve $\gamma \cap D(p,R)$ is homologically trivial in the relative homology $H_1(D(p,R),\partial D(p,R); \Z)$. The argument is the same as in the corresponding part of the proof of Theorem \ref{th:cup} but with $\Z$-coefficients.
Thus $\ell(\partial D(p,R))\geq 2 \ell(\gamma \cap D(p,R))=4R$ and, by the coarea formula,
$$
\area(S)\geq \area(D(p,L(\gamma)/2))= \int_0^{L(\gamma)\over2} \ell(\partial D(p,R)) \, dR\geq L(\gamma)^2/2\geq \ell(\gamma) \cdot L(\gamma)/2.
$$
\end{proof}

Suppose in first instance that the homological systole satisfies
$$
\sys(S)\geq \sqrt{\pi\area(S)\over g-1}.
$$
By \cite[Theorem 1.1]{BPS12} there exists a family $(\gamma_1,\ldots,\gamma_{2g})$ of loops on $S$ that generate $H_1(S;\Z)$ and whose lengths satisfy the following inequalities:
\begin{equation}\label{eq:length.basis}
\length(\gamma_{k}) \leq 2^{16} \cdot g \cdot \frac{\log(2g-k+2)}{2g-k+1} \cdot \sqrt{\area(S)\over 4 \pi (g-1)}.
\end{equation}
for $k=1,\ldots,2g$.
Thus
$$
\prod_{i=k}^{2g} \length(\gamma_{k}) \leq {2^{32g}g^{2g}\over (2g)!} \cdot \prod_{k=1}^{2g} \log(2g-k+2) \cdot \left({\area(S)\over 4 \pi (g-1)}\right)^g.
$$
From the following refinement of Stirling's formula (see~\cite[p.~54]{Feller})
$$
n!>\sqrt{2\pi n} \left(\frac{n}{e}\right)^n e^{1/(12n +1)} 
$$
we can deduce that
$$
\prod_{k=1}^{2g} \ell(\gamma_k)\leq {C^g} \cdot \left(\log 2g \right)^{2g} \cdot \left({\area(S)\over g}\right)^g
$$
with $C={2^{31} e^2 \over 4 \pi}< 2^{31}$.
In particular we get
$$
\prod_{i=k}^{2g} \ell_g(\gamma_k)\leq {C^g} \cdot \area(S)^g.
$$
Suppose now that the homological systole satisfies
$$
\sys(S)\leq \sqrt{\pi \area(S)\over g-1}
$$ 
and pick a class $\gamma$ realizing the homological systole. By Lemma \ref{lemma.inter.systole} 
$$
\ell(\gamma) \cdot L(\gamma) \leq 2 \area(S).
$$
Because $\gamma$ realizes the systole, we can find a closed geodesic $\delta$ of length $L(\gamma)$ and such that $[\gamma]\cap [\delta] =1$. Set $\gamma_{2g}=\gamma$ and $\gamma_{2g-1}=\delta$. We cut $S$ along $\gamma_{2g}$ and glue two round hemispheres as caps along the two components of the boundary $\partial \overline{(S\setminus \gamma)}$ to get a new Riemannian closed surface $S'$ of genus $g-1$. Its area satisfies 
$$
\area(S')\leq \area(S)+{\ell(\gamma)^2\over \pi}\leq {g \over g-1} \area(S)
$$
by assumption on the systole. 
By induction (and using Remark \ref{rem:prodtorus}) we know that there exists a family of closed geodesics $\gamma_1,\ldots,\gamma_{2g-2}$ on $S'$ inducing a basis of $H_1(S';\Z)$ and whose length product satisfies 
$$
\prod_{k=1}^{2g-2} \ell(\gamma_k)\leq {C^{g-1}} \cdot \area(S')^{g-1}.
$$
We can homotope all these curves in $S\setminus \gamma \subset S'$ without increasing their length. The corresponding curves in $S$ define, together with $\gamma_{2g}$ and $\gamma_{2g-1}$, a family $(\gamma_1,\ldots,\gamma_{2g})$  of closed geodesics. It is straightforward to check that this family induces a basis of $H_1(S;\Z)$ using the symplectic nature of $\cdot \cap \cdot$. Furthermore
\begin{eqnarray*}
\prod_{k=1}^{2g} \ell(\gamma_k)&\leq& 2 \area(S)\cdot {C^{g-1}} \cdot \area(S')^{g-1}\\
&\leq& 2 C^{g-1} \left({g \over g-1}\right)^{g-1} \area(S)^g\\
&\leq &2e \cdot {C^{g-1}} \cdot \area(S)^g
\end{eqnarray*}
which concludes the argument. 
\end{proof}

Observe that, as a byproduct of the proof, we have the following. 

\begin{proposition}\label{prod.surfaces.sys.big}
Let $S$ be an orientable Riemannian surface of genus $g$ whose homological systole satisfies
$$
\sys(S) \geq \sqrt{\area(S)\over 4\pi (g-1)}.
$$
Then there exist closed geodesics $\gamma_1,\ldots,\gamma_{2g}$ inducing a basis of $H_1(S;\Z)$ and whose length product satisfies the following inequality:
$$
\prod_{k=1}^{2g} \ell(\gamma_k)\leq C^g \cdot \left(\log 2g \right)^{2g} \cdot \left({\area(S)\over g}\right)^g
$$
with $C<2^{31}$.
\end{proposition}
 
Because 
$$
H_1(M,\Z_2) \simeq H_1(M,\Z) \otimes \Z_2,
$$ 
we deduce that the closed geodesics $\gamma_1,\ldots,\gamma_{2g}$ above also induce a $\Z_2$-homology basis, thus proving Proposition \ref{prod.surfaces.sys.big.intro}.

\begin{remark}
A similar result to Theorem \ref{prod.Z.surfaces} and Proposition \ref{prod.surfaces.sys.big} holds for non-orientable Riemannian surfaces. Indeed the analog of the upper bounds (\ref{eq:length.basis}) on the successive lengths of a $\Z$-homology basis still holds true in the non-orientable case, see \cite[Remark 2.4]{BPS12}.
\end{remark}

\medskip

\subsection{Length product of $g$ homologically independent closed geodesics}

Using the same approach as in proof of Theorem \ref{prod.Z.surfaces}, we can prove the following.

\begin{proposition}\label{surfaces.gcurves}
Let $S$ be an orientable Riemannian surface of genus $g$. 

Then there exist closed geodesics $\gamma_1,\ldots,\gamma_{g}$ whose homology classes are independent in $H_1(S;\Z)$ and whose length product satisfies the following inequality:
$$
\prod_{k=1}^{g}\ell(\gamma_k)\leq C^g \log(2g)^g \left(\frac{\area(S)}{g}\right)^{g/2}
$$
with $C<2^{31}$.
\end{proposition}

\begin{proof}The proof is quite similar to proof of Theorem \ref{prod.Z.surfaces}.
First notice that if the homological systole satisfies
$$
\sys(S) \geq \sqrt{\area(S)\over 4\pi (g-1)}
$$
the result follows from Proposition \ref{prod.surfaces.sys.big}.
Suppose now that the homological systole satisfies
$$
\sys(S)\leq \sqrt{\pi \area(S)\over g-1}
$$
and pick a class $\gamma$ realizing the homological systole. We cut $S$ along $\gamma$ and glue two round hemisphere caps to get a new Riemannian closed surface $S'$ of genus $g-1$ and area 
$$
\area(S')\leq {g \over g-1} \area(S).
$$
By induction we know that there exist a family of closed geodesics $\gamma_1,\ldots,\gamma_{g-1}$ on $S'$ which are $\Z$-homologically independent and whose length product satisfies 
$$
\prod_{k=1}^{g-1} \ell(\gamma_k)\leq C^{g-1} \log(2(g-1))^{g-1} \left(\frac{\area(S')}{g-1}\right)^{g-1\over2}.
$$
We can homotope all these curves in $S\setminus \gamma \subset S'$ without increasing their length. The corresponding curves in $S$ define together with $\gamma_g:=\gamma$ a family $(\gamma_1,\ldots,\gamma_{g})$ of $\Z$-homologically independent closed geodesics. Furthermore
\begin{eqnarray*}
\prod_{k=1}^{g} \ell(\gamma_k)&\leq& C^{g-1} \log(2(g-1))^{g-1} \left(\frac{\area(S')}{g-1}\right)^{g-1\over2} \cdot \sqrt{\pi \area(S)\over g-1}\\
&\leq& C^{g-1} \log(2g)^{g-1} \left(\left(\frac{g}{g-1}\right)^2\frac{\area(S)}{g}\right)^{g-1\over2} \cdot \sqrt{\pi g\over g-1} \cdot \sqrt{\area(S) \over g}\\
&\leq &\sqrt{2\pi} \cdot C^{g-1} \log(2g)^{g-1} \cdot \left(\frac{g}{g-1}\right)^{g-1} \cdot \left(\frac{\area(S)}{g}\right)^{g\over2}\\
&\leq&\sqrt{2\pi} \cdot e \cdot C^{g-1} \log(2g)^{g-1} \cdot \left(\frac{\area(S)}{g}\right)^{g\over2}
\end{eqnarray*}
and the result follows.
\end{proof}

The proof also works with $\Z_2$-coefficients thus proving Proposition \ref{surfaces.gcurves.intro}.

\subsection{Length product of two intersecting loops for hyperbolic surfaces} 

If Conjecture \ref{conj.surfaces.Z} is true, this would imply the following length product inequality for two intersecting closed geodesics.

\begin{conjecture}\label{conj.inter.surfaces}
There exists a constant $c>0$ such that the following holds. On any orientable closed Riemannian surface $S$ of genus $g$ there are two closed geodesics $\gamma$ and $\delta$ with $[\gamma]\cap [\delta]\neq 0$ and whose length product satisfies the following inequality:
$$
\ell(\gamma)\cdot \ell(\delta)\leq \left(c \, \log (2g) \right)^{2} \cdot {\area(S)\over g}.
$$
\end{conjecture}

Let us explain how to deduce Conjecture \ref{conj.inter.surfaces} from Conjecture \ref{conj.surfaces.Z}. If Conjecture \ref{conj.surfaces.Z} is true, $S$ contains a family $(\gamma_1,\ldots,\gamma_{2g})$ of closed geodesics which form a basis of $H_1(S;\Z)$ and with length product
$$
\prod_{k=1}^{2g} \ell(\gamma_k)\leq \left(c \log (2g) \right)^{2g} \cdot \left({\area(S)\over g}\right)^g.
$$
The intersection form on $H_1(S,;\Z)$ induces a symplectic form. This implies that we can reorder the elements of this basis such that $[\gamma_{2i-1}]Ê\cap [\gamma_{2i}]\neq 0$ for $i=1,\ldots,g$ according to the following lemma due to Gutt.

\begin{lemma}\label{LemGut}\cite{Gut18}
Let $\omega$ a symplectic form and $(e_1,\ldots,e_{2n})$ a basis of $\R^{2n}$. Then there exists a permutation $\sigma \in \mathfrak{S}_{2n}$ such that
$$
\omega(e_{\sigma(2k-1)},e_{\sigma(2k)})\neq 0
$$
for $k=1,\ldots,n$.
\end{lemma}

\begin{proof}
The result is equivalent to finding a permutation $\sigma \in {\mathfrak S}_{2n}$ such that
$$
\prod_{k=1}^{n}\omega(e_{\sigma(2k-1)},e_{\sigma(2k)})\neq 0.
$$
Recall that if $\eta$ and $\omega$ are respectively an alternating $k$-form and $m$-form then
$$
\eta \wedge \omega (e_1,\ldots,e_{k+m})=\sum_{\sigma \in \mathrm{Sh}_{k,m}} sgn(\sigma)\cdot\eta(e_{\sigma(1)},\ldots, e_{\sigma(k})\cdot \omega(e_{\sigma(k+1)},e_{\sigma(k+m)})
$$
where $\mathrm{Sh}_{k,m} \subset \Sigma_{k+m}$ is the subset of $(k,m)$ shuffles, i.e., all permutations $\sigma$ of $\{1, \ldots, k + m\}$ such that $\sigma(1)<\ldots<\sigma(k)$ and $\sigma(k+1)<\ldots<\sigma(k+m)$). Then we compute
\begin{eqnarray*}
\omega^n(e_1,\ldots,e_n)&=&\omega^{n-1} \wedge \omega (e_1,\ldots,e_{2n})\\
&=&\sum_{\sigma \in \mathrm{Sh}_{2n-2,2}} \text{sgn}(\sigma)\cdot\omega^{n-1}(e_{\sigma(1)},\ldots, e_{\sigma(2n-2)})\cdot \omega(e_{\sigma(2n-1)},e_{\sigma(2n)})\\
&=&\sum_{\sigma \in \mathrm{Sh}_{2,\ldots,2}} \text{sgn}(\sigma)\cdot \prod_{i=0}^{n-1} \omega(e_{\sigma(2i+1)}, e_{\sigma(2i+2)}).
\end{eqnarray*}
Because $\omega$ is symplectic, we know that $\omega^n (e_1,\ldots,e_{2n})\neq 0$ and thus at least one of the products is non zero. 
\end{proof}

\begin{remark}
This proof came up after a discussion between J.~Gutt and the first author about how to prove a symplectic analog of Minkowski's first theorem. Here we can see the germ of one of the mechanisms used in the proof of Theorem \ref{manifolds.Z_2} and we are very grateful to him for showing us this nice argument. 
\end{remark}

Returning to the link between Conjecture \ref{conj.inter.surfaces} and Conjecture \ref{conj.surfaces.Z}, we can reorder the family of closed geodesics such that $[\gamma_{2i-1}]Ê\cap [\gamma_{2i}]\neq 0$ for $i=1,\ldots,g$.
This implies that 
$$
\min \{\ell(\gamma) \cdot \ell(\delta) \mid [\gamma] \cap [\delta]\neq0\} \leq \left(\prod_{k=1}^{2g} \ell(\gamma_k)\right)^{1/g}\leq \left(c \, \log (2g) \right)^{2} \cdot {\area(S)\over g}.
$$

Observe that it is not possible to directly deduce Conjecture \ref{conj.inter.surfaces} from Proposition \ref{surfaces.gcurves}: the shortest $g$ homologically independent closed geodesics might lie in a Lagrangian subspace of $H_1(M;\Z)$ for the intersection form. However for hyperbolic metrics we are able establish the result.

\begin{theorem}\label{th.inter.hyp}
There exists a constant $c>0$ such that the following holds. A genus $g$ closed orientable hyperbolic surface $S_0$ contains two closed geodesics $\gamma$ and $\delta$ with $[\gamma]\cap [\delta]=1$ and whose length product satisfies the following inequality:
$$
\ell(\gamma) \cdot \ell(\delta)\leq \left(c \, \log (2g)\right)^2.
$$
\end{theorem}

Recall that $\area(S_0)=4\pi(g-1)$ as the metric is hyperbolic. Thus Theorem \ref{th.inter.hyp} etablishes Conjecture \ref{conj.inter.surfaces} for hyperbolic surfaces.

\begin{proof}
 We first collect some information about the local geometry near simple closed geodesics of small length. 
 Let $\eta$ be a simple closed geodesic. According to the collar lemma (see for instance \cite[Theorem 4.1.1]{Bus92}) the collar 
 $$
 {\mathcal C}(\eta):=\{p \in S_0 \mid \text{dist} \, (p,\eta)\leq w(\eta)\}
 $$
of width $w(\eta):=\text{arcsinh} \left({1/ \sinhÊ\left({\ell(\eta)/ 2}\right)}\right)$ around $\eta$ is isometric to the Riemannian cylinder 
 $$
( [-w(\eta),w(\eta)]\times {\mathbb S}^1,ds^2=d\rho^2+\ell(\eta)^2\cosh^2 \rho \, dt^2).
 $$
 For any $l \in ]\ell(\eta),\ell(\eta)\cosh w(\eta)]$ denote by $\eta_l^+$ and $\eta_l^-$ the two parallel curves of ${\mathcal C}(\eta)$ above and below $\eta$ with length $l$. Note that they are never geodesics but are of constant curvature. Denote by ${\mathcal C}_l(\eta)$ the collar around $\eta$ with boundary curves $\eta_l^-$ and $\eta_l^+$. With this notation $ {\mathcal C}(\eta)= {\mathcal C}_l(\eta)$ for $l=L_\eta:=\ell(\eta)\cosh w(\eta)$.

 \begin{lemma}
 There exists a positive constant $\eps_0\leq 1$ such that the following holds.

If $\eta$ is a simple closed geodesic of length less than or equal to $\eps_0$, then
\begin{enumerate}
\item $\eps_0 \leq L_\eta$;
\item the distance $d(\eta)$ between $\eta_{\eps_0}^+$ and $\eta_{L_\eta}^+$ satisfies $d(\eta)\geq L_\eta/2$.
\end{enumerate}
 \end{lemma}
 
 \begin{proof}[Proof of the lemma]
Condition (1) is equivalent to 
 $$
\eps_0\leq f(\ell):= \ell \cdot \cosh \left(\text{arcsinh} \left({1/ \sinhÊ\left({\ell/ 2}\right)}\right)\right)
 $$
 for any $\ell\leq \eps_0$. Observe that $\lim_{\ell \to 0} f(\ell)= 2$ so define $\eps_1>0$ as the largest number such that
 $$
 \min_{\ell \in [0, \eps_1]} f(\ell)\geq 1
 $$
(possibly $\eps_1=\infty$).
Condition (1) is then satisfied for any $\eps_0\leq\min \{\eps_1,1\}$.

Because
$$
d(\eta)=w(\eta)-\text{arccosh}{\eps_0\over \ell(\eta)},
$$
Condition (2) is equivalent to
$$
w(\eta)- \text{arccosh} {\eps_0 \over \ell(\eta)}\geq {L_\eta\over 4}
$$
which holds if
$$
\eps_0\leq g(\ell):=\ell \cosh \left(\text{arcsinh} \left({1/ \sinhÊ\left({\ell/ 2}\right)}\right)-{f(\ell) / 4}\right)
$$
for any $\ell\leq \eps_0$.
Again $\lim_{\ell \to 0} g(\ell)=2$
 when $\ell \to 0$. Define $\eps_2>0$ as the largest number such that
 $$
 \min_{\ell \in [0, \eps_2]} g(\ell)\geq 1
 $$
(again we allow $\eps_2=\infty$). Condition (2) is satisfied for any $\eps_0\leq\min \{\eps_2,1\}$.

Now set $\eps_0:=\min \{\eps_1,\eps_2,1\}$ and the lemma is proved.
 \end{proof}
 
If $\sys(S_0)\geq \eps_0$, then Theorem 1.1 from \cite{BPS12} guarantees the existence of a family of loops $\gamma_1,\ldots,\gamma_{2g}$ in $S_0$ that generate $H_1(S_0;\Z)$ and such that 
$$
\length(\gamma_{k}) \leq C_0 \cdot g \cdot \frac{\log(2g-k+2)}{2g-k+1}
$$
for $k=1,\ldots,2g$ and with $C_0:=2^{16} / \eps_0$. Among the $g+1$ first ones, at least two of them have non-zero intersection and their length product is bounded from above by
$$
\left(C_0 \, \log (2g)\right)^2.
$$
Furthermore they belong to a minimal length homology basis (for successive lengths), and so they pairwise intersect at most once, hence exactly once. 


\medskip

Now suppose that $\sys(S_0)\leq \eps_0$.

Let $(\eta_1,\ldots,\eta_p)$ be the set of homologically non-trivial simple closed geodesics in $S_0$ of length less than $\eps_0$ and ordered such that $\ell(\eta_1)\leq\ldots \leq \ell(\eta_p)$. As $\eps_0\leq 1$ these loops do not pairwise intersect and moreover the corresponding collars ${\mathcal C}(\eta_i)$ are pairewise disjoint for $i=1,\ldots,p$. In particular $p\leq 3g-3$. For each $i=1,\ldots,p$ we cut off the collar ${\mathcal C}_{\eps_0}(\eta_i)$ and paste $(\eta_i)^+_{\eps_0}$ with $(\eta_i)^-_{\eps_0}$. We denote by $\eta_i'$ the new curve obtained from the pasting of these two curves in the new surface denoted by $S'_0$. 

The new surface $S'_0$ is no longer hyperbolic, but does have area less than the area of $S_0$. Furthermore its homological systole satisfies the following lower bound.

\begin{lemma}
$\sys(S'_0)\geq \eps_0$.
\end{lemma}

 \begin{proof}[Proof of the lemma]
 Observe that the homological systole is realized by a simple closed curve in $S'_0$.
 
Consider a closed curve $\gamma'$ with non-zero intersection with one of the $\eta'_i$. We denote by ${\mathcal C}(\eta'_i)$ the cylinder around $\eta'_i$ of width $d(\eta_i)$. The intersection of $\gamma'$ with ${\mathcal C}(\eta'_i)$ consists of at least one arc which is homotopically non-trivial relative to the boundary of the cylinder. In particular the length of $\gamma'$ is bounded from below by $2d(\eta_i)\geq \eps_0$. 
 
Now let $\gamma'$ denotes a homologically non-trivial simple closed curve which has zero intersection with any of the curves $\eta'_i$ and whose homology class is different from any of the classes $[\eta'_i]$ for $i:=1,\ldots,p$. If $\gamma'$ does not intersect any of the $\eta'_i$'s we are done: $\gamma'$ actually coincides with a homologically non-trivial simple closed curve in $S_0$ different from any of the $\eta_i$'s so its length is at least $\eps_0$. If $\gamma'\cap \eta'_i\neq \emptyset$, consider one of the arcs of $\gamma' \cap {\mathcal C}(\eta'_i)$ intersecting $\eta'_i$. This arc can not be contained in ${\mathcal C}(\eta'_i)$ so its length is at least $2d(\eta_i)\geq \eps_0$.

Because any closed curve in the homology class of one of the $\eta'_i$'s has length at least $\eps_0$, we conclude that $\sys(S'_0)\geq \eps_0$.
 \end{proof}
 
By applying once again Theorem 1.1 from  \cite{BPS12}, we can find two homologically independent closed curves $\gamma'$ and $\delta'$ of $S'_0$ with non-zero intersection and whose length product is bounded above by $\left( C_0 \, \log (2g)\right)^2$. Moreover, we can assume that $\gamma'$ and $\delta'$ minimize the length product over pairs of closed simple curves with non-zero intersection number.

\medskip

Suppose first that both $\gamma'$ and $\delta'$ have zero intersection number with any of the $\eta_i'$'s. This implies that $\gamma'$ (respectively $\delta'$) is disjoint from all of the $\eta_i'$'s or coincides with one of the $\eta_i$'s.

To see this suppose that $\gamma' \cap \eta_i\neq \emptyset$ for some $i$ while $\gamma' \neq \eta'_i$. There exist two points $p$ and $q$ of $\gamma' \cap \eta_i$ and a subarc $\zeta$ of $\gamma'$ from $p$ to $q$ such that 
\begin{itemize}[label=$\cdot$,itemsep=1ex]
\item the interior of $\zeta$ does not intersect $\eta'_i$;
\item $\zeta$ starts and ends on the same side of $\eta'_i$.
\end{itemize}
Thus $\gamma'$ decomposes into the sum of two simple loops $\gamma'_1$ and $\gamma'_2$ respectively made of the concatenation of $\zeta$ and of $\gamma' \setminus \zeta$ with the shortest path of $\eta'_i$ from $p$ to $q$. Observe that
$$
\ell(\gamma'_i)< \ell(\gamma')
$$
as $\ell(\eta'_i)/2=\eps_0/2\leq d(\eta_i)$.

Now because $[\gamma']=[\gamma'_1]+[\gamma'_2]$ one of the curves $\gamma'_1$ or $\gamma'_2$ has non-zero intersection number with $\delta'$ and the product of its length with $\ell(\delta')$ is strictly less than the initial length product $\ell(\gamma') \cdot \ell(\delta')$. This is a contradiction and we can thus assume that both $\gamma'$ and $\delta'$ are either disjoint from any of the $\eta_i'$'s or coincide with one of them. 

Because $\gamma'$ and $\delta'$ are either disjoint from any of the $\eta_i'$'s or coincide with one of them, they actually coincide with closed curves in $S_0$, and so we have found a pair of curves with the desired properties. 

\medskip

Suppose now that $\gamma'$ (after possibly commuting $\gamma'$ and $\delta'$) has non-zero intersection number with one of the $\eta'_i$'s. Let $i_0 \in \{1,\ldots,p\}$ be the smallest integer $i$ between $1$ and $p$ and such that $[\gamma'] \cap [\eta'_i] \neq 0$. 

The curve $\gamma'$ lifts to a multiarc of $S_0$ with extremal points belonging to $\cup_{i=1}^p \partial C_{\eps_0}(\eta_i)$. This multiarc can be completed into a closed curve $\gamma$ on $S_0$ by joining two extremal points lying in the boundary of $\partial C_{\eps_0}(\eta_i)$ and corresponding to the same point of $\gamma'$ by an arc crossing the cylinder $C_{\eps_0}(\eta_i)$. The length of each such arcs is bounded from above by 
$$
2 \, \text{arcsinh}\left({1\over\sinh (\ell(\eta_{i_0})/2)}\right).
$$
Because $\gamma'$ is minimizing each arc cutting $\eta'_i$ must leave the cylinder ${\mathcal C}(\eta'_i)$, and the number of such arcs is bounded from above by 
$$
\ell(\gamma')/({2\eps_0}).
$$
Thus 
$$
\ell(\gamma) \leq \ell(\gamma')\left(1+{1\over \eps_0} \cdot \text{arcsinh}\left({1\over\sinh (\ell(\eta_{i_0})/2)}\right)\right).
$$
Then we get
$$
\ell(\gamma) \cdot \ell(\eta_{i_0}) \leq \ell(\gamma')\cdot \ell(\eta_{i_0})+\ell(\gamma')\cdot \ell(\eta_{i_0}) \cdot {1\over \eps_0} \cdot \text{arcsinh}\left({1\over\sinh (\ell(\eta_{i_0})/2)}\right)
$$
and because
$$
\ell \cdot \text{arcsinh}\left({1\over\sinh (\ell/2)}\right)
 \to 0
$$
when $\ell \to 0$, there exists a universal constant $c_o$ such that
$$
{1\over \eps_0} \cdot \ell \cdot \text{arcsinh}\left({1\over\sinh (\ell/2)}\right)
 \leq c_0
$$
for any $\ell \leq \eps_0$.
Thus
$$
\ell(\gamma) \cdot \ell(\eta_{i_0})\leq \ell(\gamma')\cdot \ell(\eta_{i_0})+c_0 \cdot \ell(\gamma')\leq (1+c_0/\eps_0) \cdot \ell(\gamma') \cdot \ell(\delta')
$$
as $\ell(\delta')\geq \sys(S'_0)\geq \eps_0$.
To conclude, remark that $\gamma$ and $\eta_{i_0}$ are two closed curves with non-zero intersection number since $[\gamma]\cap [\eta_{i_0}]=[\gamma']\cap [\eta'_{i_0}]$. Furthermore their length product is bounded from above as follows:
$$
\ell(\gamma) \cdot \ell(\eta_{i_0})\leq (1+c_0/\eps_0) (C_0 \log (2g))^2.
$$
This proves the result.
\end{proof}



\end{document}